\documentclass[11pt, oneside, letterpaper]{article}
\usepackage{amsmath}
\usepackage{amsthm}
\usepackage{amsfonts}
\usepackage{amssymb}
\usepackage{graphicx}
\usepackage{color}
\usepackage{rotating}
\usepackage{url}
\usepackage{xspace}
\usepackage{graphicx}
\usepackage{caption}
\usepackage{subcaption}
\usepackage{tikz}
\usepackage{authblk}
\newcommand{\R}{\mathbb R}

\newcommand{\Z}{\mathbb Z}

\DeclareMathOperator{\argmax}{argmax}

\newcommand{\conv}{\mathop{\mathrm{conv}}}

\newtheorem{proposition}{Proposition}
\newtheorem{remark}{Remark}
\newtheorem{definition}{Definition}
\newtheorem{observation}{Observation}

\newtheorem{theorem}{Theorem}

\newcounter{mynotes}
\newcommand{\remove}[1]{}

\title{The Strength of Multi-row Aggregation Cuts for Sign-pattern Integer Programs}

\author[1]{Santanu S. Dey\thanks{santanu.dey@isye.gatech.edu}}
\author[1]{Andres Iroume\thanks{airoume.gatech@gmail.com}}
\author[1]{Guanyi Wang\thanks{gwang93@gatech.edu}} 
\affil[1]{\small School of Industrial and Systems Engineering, Georgia Institute of Technology} 
\begin{document}

\maketitle

\begin{abstract}
In this paper, we study the strength of \emph{aggregation cuts} for \emph{sign-pattern} integer programs (IPs). Sign-pattern IPs are a generalization of packing IPs and are of the form $\{x\in \Z^n_+\ | \ Ax\le b\}$ where for a given column $j$, $A_{ij}$ is either non-negative for all $i$ or non-positive for all $i$. Our first result is that the \emph{aggregation closure} for such sign-pattern IPs can be {\it 2-approximated} by the \emph{original 1-row closure}. This generalizes a result for packing IPs from \cite{bodur2016aggregation}. On the other hand, unlike in the case of packing IPs, we show that the \emph{multi-row aggregation closure} cannot be well approximated by the \emph{original multi-row closure}. Therefore for these classes of integer programs general aggregated multi-row cutting planes can perform significantly better than just looking at cuts from multiple original constraints.
\end{abstract}

\section{Introduction}

In a recent paper \cite{bodur2016aggregation}, Bodur et al. studied the strength of \emph{aggregation cuts}. An aggregation cut is obtained as follows: (i) By suitably weighing and adding the constraints of a given integer programming (IP) formulation, one can obtain a  relaxation which is defined by a single constraint together with variable bounds. (ii) All the valid inequalities for the integer hull of this knapsack-like set are called as aggregation cuts. The set obtained by adding all such aggregation cuts (for all possible aggregations) is called the \emph{aggregation closure}. Such cuts are commonly used in practice by state-of-the-art solvers \cite{wolsey1975faces,zemel1978lifting,weismantel19970,marchand2001aggregation,fukasawa2011exact} and also have been studied in theory~\cite{marchand2001aggregation,andersen2010zero,DeyR08}. A very special subclass of the aggregation cuts are the cuts obtained from original constraints of the formulation as the knapsack-like relaxation (i.e. no aggregation at all). The weaker closure obtained from such cuts is called as the \emph{original 1-row closure} \cite{bodur2016aggregation}.  The paper \cite{bodur2016aggregation} shows that for packing and covering IPs, the aggregation closure can be {\it 2-approximated} by the original 1-row closure. In contrast, they show that for general IPs, the aggregation closure can be arbitrarily stronger than the original 1-row closure. 

The aggregation cuts are obtained based on our ability to generate valid inequalities for feasible sets described by one non-trivial constraint and variable bounds. Recently there has a large body of work on multi-row cuts, for example~\cite{DeyR08,AndersenLWW07,DeyW10,LouveauxW08,CornuejolsM09,BorozanC09,BasuHK17,BasuHKM13,DeyR10,AverkovB15}. Also see the review articles~\cite{RichardD10,BasuHK16a,BasuHK16} and analysis of strength of these cuts~\cite{BasuBCM09,AverkovBP17, bodur2017lower}. Therefore, it is natural to consider the notion of multi-row aggregation cuts. Essentially by using $k$ different set of weights on the constraints of the problem one can produce a relaxation that involves $k$ constraints together with variable bounds. We call the valid inequalities for the integer hull of such relaxations as $k$-row or multi-row aggregation cuts. Analogous to the case of aggregation cuts, we can also define the notion of $k$-row aggregation closure and the original $k$-row closure~\cite{bodur2016aggregation} (i.e. generate cuts from all relaxations described by $k$ constraints from the IP formulation). The results in~\cite{bodur2016aggregation} can be used to show that for packing and covering IPs the $k$-row aggregation closure can be approximated by the original $k$-row closure within a multiplicative factor that depends only on $k$. (We obtain sharper bounds for the case of $k = 2$ in this paper.)

For packing and covering IPs all the coefficients of all the variables in all the constraints have the same sign. Therefore when we aggregate constraints we are not able to ``cancel'' variables, i.e., the support of an aggregated constraint is exactly equal to the union of supports of the original constraints used for the aggregation. A natural conjecture for the fact that the aggregation closure (resp. multi-row aggregation closure) is well approximated by the original 1-row closure (resp. original multi-row closure) for packing and covering problems, is the fact that such cancellations do not occur for these problems. Indeed one of the key ideas used to obtain good candidate aggregations in the procedure described in \cite{marchand2001aggregation} is to use aggregations that maximize the chances of cancellation. Also see~\cite{andersen2010zero}, which discusses the strength of split cuts obtained from aggregations that create cancelations. 

In order to study the effect of cancellations, we study the strength of aggregation closures vis-\`a-vis original row closures for \emph{sign-pattern IPs}. A sign-pattern IP is a problem  of the form $\{x\in \Z^n_+\ | \ Ax\le b\}$ where a given variable has exactly the same sign in every constraint, i.e. for a given column $j$, $A_{ij}$ is either non-negative for all rows $i$ or non-positive for all rows $i$. Thus aggregations do not create cancellations. 

Our study reveals interesting results. On the one hand we are able to show that the aggregation closure for such sign-pattern IPs is {\it 2-approximated} by the original 1-row closure, supporting the conjecture that non-cancellation implies that aggregation is less effective. On the other hand, unlike packing and covering IPs, the multi-row aggregation closure cannot be well approximated by the original multi-row closure. So for these classes of problems, multi-row cuts may do significantly better than single-row cuts, especially those obtained by aggregation.

The structure of the rest of the paper is as follows. In Section \ref{sip:def} we provide formal definitions and statements of our results. In Section \ref{spi:proofs} we present the proofs for our results.

\section{Definitions and statement of results}\label{sip:def}

\subsection{Definitions}

For an integer $n$, we use the notation $[n]$ to describe the set $\left\{ 1,\ldots,n \right\}$ and for $k\le n$ non-negative integer, we use the notation ${[n] \choose k}$ to describe all subsets of $[n]$ of size $k$. For $i \in [n]$, we denote by $e_i$ the $i^{th}$ vector of the standard basis of $\R^n$. The convex hull of a set $S$ is denoted as $\conv (S)$. For a set $S\subset \R^n$ and a positive scalar $\alpha$ we define $\alpha S := \left\{ \alpha u\ |\ u \in S  \right\} $. We use $P^I$ to denote the convex hull of integer feasible solutions of $P$ (i.e. the integer hull of $P$). For a given linear objective function, we let $z^{LP}$, $z^{IP}$ denote the optimal objective function over $P$ and $P^I$ respectively. 

\subsubsection{Sign-pattern IPs}

\begin{definition}
Let $n$ be an integer, let $J^+,J^-\subset [n]$ such that $J^+ \cap J^- = \emptyset$ and $J^+ \cup J^- = [n]$.
We call a polyhedron $P$, a $(J^+,J^-)$ sign-pattern polyhedron if it is of the form
\begin{align*}
P &= \left\{  x\in \R^n_+  \  | \   \sum_{j \in J^+} A_{ij} x_j  - \sum_{j \in J^-} A_{ij} x_j \le b_i \quad \forall i \in [m]  \right\},
\end{align*}
where $A_{ij} ,b_i\ge 0,\ \forall i \in [m], \forall j\in [n]$. Additionally, we require $A_{ij}\le b_i\ \forall j \in J^+,\ \forall i\in [m]$.
\end{definition}

Through out this paper, whenever we refer to a packing polyhedron or a covering polyhedron, we make a similar assumption on the coefficients defining the polyhedron, i.e., if $P:= \{x \in \mathbb{R}^n_{+} | Ax \leq b\}$ is a packing (resp. $P:= \{x \in \mathbb{R}^n_{+}| Ax \geq b\}$ is a covering) polyhedron we assume $A_{ij} \leq b_i$ for all $i \in [2],\ j \in [n]$. Also we assume that all data is rational.

\begin{definition}\label{def:alphapack}
Given two polyhedra $P$ and $Q$ contained in $\R^n_+$ such that $P\supseteq Q \supseteq \{0\}$, and a positive scalar $\alpha \geq 1$, we say that $P$ is an $\alpha$-approximation of $Q$ if $P \subset \alpha  Q$.
\end{definition}

\begin{remark}\label{rem:alphapack}
Let $P\supseteq Q \supseteq \{0\}$. Suppose we are maximizing a linear objective over $P$ and $Q$ and let the optimal objective function over $P$ and $Q$ be $z^P$ and $z^Q$ respectively. Then $P$ being an $\alpha$-approximation of $Q$ implies that either $z^P = z^Q = \infty$ or $z^P \geq z^Q \geq \frac{1}{\alpha}\cdot z^P$. Therefore, in order to show $P$ is not an $\alpha$-approximation of $Q$, all we need to do is to establish that there is a linear objective such that $z^P, z^Q < \infty$ and $z^Q < \frac{1}{\alpha}\cdot z^P$.
\end{remark}

\begin{remark}
Definition \ref{def:alphapack} does not hold for covering polyhedron. We will refer to $\alpha$-approximation results for covering polyhedron for comparison. Given two covering polyhedra $P$ and $Q$ such that $P\supseteq Q$, and a positive scalar $\alpha \geq 1$, we say that $P$ is an $\alpha$-approximation of $Q$ if $P \subset \frac{1}{\alpha} Q$. This is equivalent to saying that if we are minimizing a non-negative linear function over $P$ and $Q$ (and the optimal objective function over $P$ and $Q$ are $z^P$ and $z^Q$ respectively), then $z^P \leq z^Q \leq \alpha\cdot z^P$.

\end{remark}
\subsubsection{Closures}

Given a polyhedron $P$, we are interested in cuts for the pure integer set $P\cap \Z^n$. 

\begin{definition}
For $P=\left\{ x\ge 0 \ | \  Ax\le  b  \right\}$, where $A \in \mathbb{R}^{m \times n}$, $k\ge 1$ integer, and $\lambda_1,\ldots,\lambda_k \in \mathbb{R}^m_+$, let
\begin{align*}
P(\lambda_1,\ldots,\lambda_k) &:=\left\{ x\ge 0\ | \ {\lambda_1}  Ax\le  {\lambda_1} b,\ldots,  {\lambda_k}  Ax\le  {\lambda_k} b  \right\}. \\
P^I(\lambda_1,\ldots,\lambda_k ) &:= \conv \left( \left\{ x\in \Z^n_+\ | \  {\lambda_1}  Ax\le  {\lambda_1} b,\ldots, {\lambda_k}  Ax\le  {\lambda_k} b  \right\} \right).
\end{align*}
\end{definition}

\begin{definition}[Closures]
Given a polyhedron $P=\left\{ x\in \R^n_+ \ |\ Ax\le b \right\}$, where $A \in \mathbb{R}^{m \times n}$, we define its aggregation closure $\mathcal{A}(P)$ as
\begin{align*}
\mathcal{A}(P) = \bigcap\limits_{\lambda \in \R^m_+} P^I(\lambda).
\end{align*}

We can generalize the \emph{aggregation-closure} to consider simultaneously $k$ aggregations, where $k \in \Z$ and $k \ge 1$. More precisely, for a polyhedron $P$ the \emph{k-aggregation} closure is defined as
\begin{align*}
\mathcal{A}_k(P) &:= \bigcap_{\lambda_1,\ldots, \lambda_k \in \R^m_+} P^I(\lambda_1,\ldots, \lambda_k).
\end{align*}

Similarly, the original 1-row closure $\it{1}\textnormal{-}\mathcal{A}(P)$ is defined as
\begin{align*}
\it{1}\textnormal{-}\mathcal{A}(P) &:=  \bigcap_{ i \in [m] }  P^I(e_{i} ).
\end{align*}

We can generalize the original 1-row closure, to the original $k$-row closure $k\textnormal{-}\mathcal{A}(P)$. More precisely, for a polyhedron $P$ the original $k$-row closure is defined as
\begin{align*}
k\textnormal{-}\mathcal{A}(P) &:=  \bigcap_{\left\{i_1,\ldots,i_k\right\}\in {[m] \choose k} }  P^I(e_{i_1},\ldots,e_{i_k} ).
\end{align*}
\end{definition}

Given a linear objective function, we let $z^{\mathcal{A}}$, $z^{\mathcal{A}_k}$, $z^{\it{1}\textnormal{-}\mathcal{A}(P)}$, and $z^{k\textnormal{-}\mathcal{A}(P)}$ be the optimal objective function over $\mathcal{A}$, $\mathcal{A}_k$, $\it{1}\textnormal{-}\mathcal{A}(P)$ and $k\textnormal{-}\mathcal{A}(P)$ respectively.

\subsection{Statement of results}

The first result compares the aggregation closure with the LP relaxation of a $(J^+,J^-)$ \emph{sign-pattern} polyhedron.

\begin{theorem}
\label{thm:ag}
For a $(J^+,J^-)$ sign-pattern polyhedron $P$, we have that $\mathcal{A}(P)$ can be 2-approximated by $P$, and thus by $1\textnormal{-}\mathcal{A}(P)$.
\end{theorem}

This result generalizes the result obtained in \cite{bodur2016aggregation} for the case of packing IPs. Since the ratio of $2$ is already known to be tight for the case of packing instances \cite{bodur2016aggregation}, the result of Theorem~\ref{thm:ag} is tight.

Next we show that, in general, for $(J^+,J^-)$  \emph{sign-pattern} IPs, the \emph{aggregation-closure} does not do a good job at approximating the $2$-aggregation closure.

\begin{theorem}
\label{thm:bad}
There is a family of $(J^+,J^-)$ sign-pattern polyhedra for which $\mathcal{A}$ is an arbitrarily bad approximation to $\mathcal{A}_2$, i.e. for each $\alpha > 1$, there is a $(J^+,J^-)$ sign-pattern polyhedron $P$ such that $\mathcal{A}(P)$ is not an $\alpha$-approximation of $\mathcal{A}_2(P)$. In contrast, if $P$ is a packing (resp. covering) polyhedron, then $\mathcal{A}(P) \subseteq 3(\mathcal{A}_2(P))$ (resp. $\mathcal{A}(P) \subseteq \frac{1}{2.5}(\mathcal{A}_2(P))$).
\end{theorem}

The previous result shows that for non-trivial $(J^+,J^-)$ \emph{sign-pattern} polyhedra, using multiple row cuts can have significant benefits over one row cuts.  This is different than for the case of packing/covering problems (where the improvement is bounded).

The next result shows that  the \emph{aggregation-closure} considering simultaneously $2$ aggregations ($\mathcal{A}_2$) can be arbitrarily stronger than the original $2$-row closure ($2\textnormal{-}\mathcal{A}$).

\begin{theorem}
\label{thm:bad2}
There is a family of $(J^+,J^-)$ sign-pattern polyhedra with 4 constraints for which $\it{2}\textnormal{-}\mathcal{A}$ is an arbitrarily bad approximation to $\mathcal{A}_2$, i.e. for each $\alpha > 1$, there is a $(J^+,J^-)$ sign-pattern polyhedron $P$ such that $\it{2}\textnormal{-}\mathcal{A}(P)$ is not an $\alpha$-approximation of $\mathcal{A}_2(P)$. In contrast, if $P$ is a packing (resp. covering) polyhedron, then $2\textnormal{-}\mathcal{A}(P)  \subseteq 3(\mathcal{A}_2(P))$ (resp. $2\textnormal{-}\mathcal{A}(P)  \subseteq \frac{1}{2.5}(\mathcal{A}_2(P))$).
\end{theorem}

The previous results establish the comparison between the sets $\it{1}\textnormal{-}\mathcal{A}(P)$ and $\mathcal{A}(P)$, $\mathcal{A}(P)$ and $\mathcal{A}_2(P)$, and $2\textnormal{-}\mathcal{A}(P)$ and $\mathcal{A}_2(P)$. These results are presented in Table \ref{tab:cont}. 

\begin{table}[ht]
\begin{center}
\caption{Upper bound (lower bound for covering case) on $\alpha$ for various containment relations; m is the number of constraints.}
\begin{tabular}{llccccc}
\hline
&  Packing & Covering & Sign-pattern  \\
\hline
 \vspace{0.3cm}
$1\textnormal{-}\mathcal{A}(P) \subset \alpha \mathcal{A}(P)  $ 								 & $\alpha \leq 2$ (\cite{bodur2016aggregation})		 & $\alpha \geq \frac{1}{2}$ (\cite{bodur2016aggregation})			& $\alpha \leq 2$ \\
 \vspace{0.1cm}
$\mathcal{A}(P) \subset \alpha  \mathcal{A}_2(P)$ 	 
& 
$\alpha \leq 3$ \ if $m\ge 2$	 
& 
$\alpha \geq \frac{1}{2.5}$ \ if $m\ge 2$ 	
&  	
$\alpha \leq \infty$ \quad if $m\ge 2$\\ 
 \\
$2\textnormal{-}\mathcal{A}(P)  \subset \alpha \mathcal{A}_2(P) $  
& 
$\alpha \leq \left\{ \begin{tabular}{ll} 
$1$ & if $m=2$\\
$3$  & if $m\ge 3$ \\ 
\end{tabular} \right.$ 		 
& 
$\alpha \geq \left\{ \begin{tabular}{ll} 
$1$ & if $m=2$\\
$\frac{1}{2.5}$  & if $m\ge 3$ \\ 
\end{tabular} \right.$ 
	 	
& 
$\alpha\leq \left\{ \begin{tabular}{ll} 
$1$ & if $m=2$\\
$?$ & if $m=3$\\
$\infty$ & if $m\ge 4$\\ 
\end{tabular} \right.$ 	 \\
\hline
\end{tabular}

\label{tab:cont}
\end{center}
\end{table}

%
%

\section{Proofs} \label{spi:proofs}

\subsection{Proof of Theorem \ref{thm:ag}}

First, we need some general properties for $(J^+,J^-)$ \emph{sign-pattern} LPs. 

\begin{proposition}\label{prop:one}
Consider a $(J^+,J^-)$ \emph{sign-pattern} polyhedron defined by one non-trivial constraint 
$P=\left\{x\ge 0\ | \ \sum_{j \in J^+}a_jx_j - \sum_{j \in J^-}a_jx_j  \le b\right\}$ and let $c$ be a vector with the same sign-pattern, i.e. $c_j\ge 0\ \forall j \in J^+$ and $c_j\le 0\ \forall j\in J^-$.  Then:
\begin{enumerate}
\item $z^{LP}=\max\limits_{x\in P} c^\top x$ is bounded if and only if $\max\limits_{j \in J^+} \frac{c_j}{a_j} \le \min\limits_{j \in J^-} \frac{-c_j}{a_j}$.
\item If $z^{LP}$ is bounded, then there exists an optimal solution $x^{LP}$ such that $x^{LP}_j = b/a_j$ for $j \in \argmax_{j\in J^+} c_j/a_j $ and $x^{LP}_k=0$ for $k\in [n]\backslash \left\{j\right\}$.
\item If $z^{LP}$ is bounded, then $z^{LP}\le 2z^{IP}$, where $z^{IP}=\max\limits_{x\in P^I} c^\top x$.
\end{enumerate}
\end{proposition}

\begin{proof}
Clearly $0\in P$, thus the $(J^+,J^-)$ \emph{sign-pattern} LP cannot be infeasible. Consider its dual
$$\textup{min}\{ by \ | \  a_j y \ge c_j \    \forall j \in J^+, \ a_j y \le -c_j \    \forall j \in J^-, \  y\ge 0\},$$
which is feasible if and only if $\max\limits_{j \in J^+} \frac{c_j}{a_j} \le \min\limits_{j \in J^-} \frac{-c_j}{a_j}$.

If $z^{LP}$ is bounded, then there exists an optimal solution that is an extreme point. Since the problem is defined by a single non-trivial constraint, each extreme point can have at most one non-zero coefficient, thus a maximizer $x^{LP}$ over the set of extreme points must be of the form $x^{LP}_{j^*}=b/a_{j^*}\ge 1$ for some $j^*\in J^+$ and  $x^{LP}_{j}=0\ \forall j \in [n]\backslash \left\{j^*\right\}$.

Clearly $\lfloor{x^{LP}\rfloor}\in P^I$, thus $\frac {z^{LP} }{ z^{IP} }\le \frac{  b/a_{j^*}  }{    \lfloor{ b/a_{j^*}  \rfloor}     }$. Finally, since $P$ is a $(J^+,J^-)$ \emph{sign-pattern} polyhedron $b/a_{j^*} \ge 1$ and thus $\frac {z^{LP} }{ z^{IP} }\le 2$.
\end{proof}


\begin{proposition}
Consider a $(J^+,J^-)$ \emph{sign-pattern} polyhedron $P$, then $P^{I}$ is also a $(J^+,J^-)$ sign-pattern polyhedron.
\end{proposition}

\begin{proof}
First, since $P$ is a  $(J^+,J^-)$ sign-pattern polyhedron, $0,e_1,\ldots,e_n\in P$, then $P^I$ is a non-empty full-dimensional polyhedron. We show that for every non-trivial facet $ax\le b$ (that is facets other than $x_j \geq 0$), we must have $a_j\ge 0$ for $j\in J^+$ and $a_j\le 0$ for $j\in J^-$.

For $j\in J^-$. Note that the recession cone of $P^I$ is the same as the recession cone of $P$. Then for every facet $ax\le b$ we have $a_j\le 0$ for $j \in J^-$ (otherwise $e_j$ would not be in the recession cone of $P^I$). 

For $j \in J^+$, assume that there exists a facet $ax\le b$ s.t. $a_j<0$. Consider $a'=a-a_j e_j$ (we zero out the $j$-th component), if $a'x\le b$ is valid, it corresponds to a stronger non-trivial constraint than $ax\le b$. In order to show that it is valid, assume that there exists $x\in P\cap \Z^n$ s.t. $a'x>b$. Note that $x_j\ge 1$ (since otherwise $a'x=ax\le b$). Consider $x'=x-x_je_j$ (clearly $x'\in P\cap \Z^n$), then $b<a'x=ax' \le b$, a contradiction.
\end{proof}

\begin{proposition}
\label{sp_one}
Let $P$ be a  $(J^+,J^-)$ sign-pattern polyhedron defined by one constraint, then $P\subset 2P^{I}$.
\end{proposition}
\begin{proof}
Assume by contradiction that there exists $x' \in \frac{1}{2}P$ s.t. $x' \notin P^I$. Since $P^I$ is a $(J^+,J^-)$ \emph{sign-pattern} polyhedron, each non-trivial facet-defining inequality $ax\le b$ satisfies $a_j\ge 0\ \forall j \in J^+$ and $a_j\le 0\ \forall j \in J^-$. Since $x'\notin P^I$, for one of these facets, we have: $ax'>b$. Now, if we consider $a$ as an objective: $\max_{x\in P^I} ax \le b$ and thus defines a bounded problem. 

Since the IP is feasible and bounded and $P$ is defined by rational data, the LP is also bounded (see \cite{Meyer1974}). Hence by Proposition \ref{prop:one}, $ax'\le \frac{1}{2} z^{LP} \le z^{I} \le b  $ a contradiction.
\end{proof}

\begin{observation}
\label{obs:prevpaper}
Let $\left\{S_i\right\}_{i\in I}$ be a collection of subsets in $\mathbb{R}^n$ and let $\alpha \in \mathbb{R}_{+}$. Then $\alpha\left(\cap_{i \in I} S_i\right) = \cap_{i\in I} \alpha(S_i)$.
\end{observation}

Now, we prove Theorem \ref{thm:ag}.

\begin{proof}
By definition, we have that $P \subset P(\lambda),\ \forall\lambda \in \mathbb{R}^m_+$. Since $ P(\lambda)$ corresponds to a $(J^+,J^-)$ \emph{sign-pattern} polyhedron defined by one constraint, by Proposition \ref{sp_one}, we have that $ P(\lambda) \subset 2 P^I(\lambda)$. 
Then taking intersection over all $\lambda \in \R^m_+$ and by Observation \ref{obs:prevpaper} we have
\begin{align*}
P \subset \bigcap_{\lambda \in \R^m_+}  P(\lambda) &\subset \bigcap_{\lambda \in \R^m_+}  2P^I(\lambda) = 2 \bigcap_{\lambda \in \R^m_+}  P^I(\lambda) =2\mathcal{A}(P).
\end{align*}

Since $1\textnormal{-}\mathcal{A}(P)$ is contained in $P$, we have that $1\textnormal{-}\mathcal{A}(P)$ is a {\it 2-approximation} of $\mathcal{A}(P)$.
\end{proof}

\subsection{Proof of Theorem \ref{thm:bad}}
In order to prove the first part of Theorem \ref{thm:bad}, consider the following family of  $(J^+,J^-)$ \emph{sign-pattern} polyhedra and $M\ge 2$ integer
\begin{eqnarray}\label{eq:p1}
\begin{array}{rll}
\max &\quad x_1 - (M-1)x_2 &  \\
\textup{s.t.} &\quad x_1 - M(M-1) x_2 &\le 1 \\
&\quad x_1 \qquad\qquad\quad & \le M+1 \\
&\quad x_1,x_2\ge 0.
\end{array} 
\end{eqnarray}
Note that the only integral solutions are $(0,0)$, $(1,0)$ and $(\bar{x}_1,\bar{x}_2)$, where $\bar{x}_1,\bar{x}_2\in \Z_+$, $\bar{x}_1\le M+1$ and $\bar{x}_2\ge 1$. Therefore, $z^{IP} = 2$ and corresponds to $(M+1,1)$. 
Additionally, $z^{LP}=M$ since the feasible point $\left(M+1,\frac{1}{M-1}\right)$ achieves that objective value and by multiplying the first constraint by $\frac{1}{M}$, multiplying the second constraint by $\frac{M - 1}{M}$ and addition them we obtain the valid inequality: $x_1 - (M-1)x_2 \leq M$.
Thus, in this case we have that: $\frac{z^{LP}}{z^{IP}} = \frac{M}{2}$.
Trivially, since we have only two constraints $\mathcal{A}_2(P) = P^I$. Now, we present our proof of the first part of Theorem \ref{thm:bad}.

\begin{proof}
 It follows from Theorem \ref{thm:ag} that for any $(J^+,J^-)$ \emph{sign-pattern} polyhedron $\frac{z^{LP}}{z^{\mathcal{A}}} \in [1,2] $. For the family of $(J^+,J^-)$ \emph{sign-pattern} polyhedra (\ref{eq:p1}) we have that $z^{IP}=z^{\mathcal{A}_2}$ and therefore
\begin{align*}
\frac{z^{\mathcal{A}}}{z^{\mathcal{A}_2}} = \frac{z^{LP}}{z^{IP}} \cdot \frac{z^{\mathcal{A}}}{z^{LP}}  \ge \frac{M}{2} \frac{1}{2}. 
\end{align*}
Since $M$ can be arbitrarily large, $\mathcal{A}(P)$ cannot be an {\it$\alpha$-approximation} of $\mathcal{A}_2(P)$ for any finite value of $\alpha$.
\end{proof}

In order to prove the second part of Theorem \ref{thm:bad} we need the following result regarding packing and covering integer programs.  
\begin{proposition}\label{prop:packcover}
Let $P:= \{x \in \mathbb{R}^n_{+} | Ax \leq b\}$ be a packing (resp. $P:= \{x \in \mathbb{R}^n_{+}| Ax \geq b\}$ be a covering) polyhedron defined by two non-trivial constraints, i.e. $A \in \mathbb{R}^{2 \times n}$, satisfying the property $A_{ij} \leq b_i$ for all $i \in [2],\ j \in [n]$. Then $P \subseteq 3P^I$ (resp. $P\subseteq \frac{1}{2.5}P^I$).
\end{proposition}
A proof of Proposition \ref{prop:packcover} is presented in the Appendix. Now we present a proof of the second part of Theorem~\ref{thm:bad}.
\begin{proof}
Let $P$ be a packing polyhedron. We have that
\begin{align*}
\mathcal{A}(P) = \bigcap_{\lambda_1 \in \R^m_+}  P^I(\lambda_1)  &\subset \bigcap_{\lambda_1 \in \R^m_+}  P(\lambda_1) \subset \bigcap_{\lambda_1,\lambda_2 \in \R^m_+}  P(\lambda_1,\lambda_2 ) \subset 3  \bigcap_{\lambda_1,\lambda_2 \in \R^m_+}  P^I(\lambda_1,\lambda_2 ) = 3 \mathcal{A}_2(P),
\end{align*}
where the third containment follows from Proposition~\ref{prop:packcover}.

Using Proposition~\ref{prop:packcover}, the proof of the covering case is the similar to the packing case. 
\end{proof}

\subsection{Proof of Theorem \ref{thm:bad2}}

In order to prove the first part of Theorem \ref{thm:bad2}, we introduce the following family of instances. For $M\ge 2$ an even integer:
\begin{align}
\max &\quad x_1 - \frac{M}{2}x_2 - \frac{M}{2}x_3 - \frac{M}{2}x_4  \label{eq:c0}\\
\textup{s.t.} &\quad x_1 - M x_2- M x_3 \qquad\qquad\le 1 \label{eq:c1} \\
&\quad x_1 - M x_2\qquad\qquad- M x_4 \le 1 \label{eq:c2}\\
&\quad x_1 \qquad\qquad- M x_3- M x_4 \le 1 \label{eq:c3}\\
&\quad x_1  \qquad\qquad\qquad\qquad\quad  \quad \ \le M +1 \label{eq:c4}\\
&\quad x_1,x_2,x_3,x_4\ge 0.  \nonumber
\end{align}
It is not difficult to verify that $z^{IP}=1$. The point $(1,0,0,0)$ has value 1 and for any feasible solution such that $x_1\ge 2$, we must have $x_2+x_3+x_4\ge 2$ (from constraints $(\ref{eq:c1})-(\ref{eq:c3})$) thus the objective function value in this case is at most 1.
Additionally, $z^{LP}=\frac{M}{4}+1$ since the feasible point $\left( M+1,\frac{1}{2},\frac{1}{2},\frac{1}{2} \right)$ achieves that objective value and by aggregating constraints $(\ref{eq:c1})-(\ref{eq:c4})$ and dividing by 4, we obtain the valid inequality: $x_1 - \frac{M}{2}x_2 - \frac{M}{2}x_3 - \frac{M}{2}x_4\le \frac{M}{4}+1$. 

Now, the proof of the first part of Theorem \ref{thm:bad2}.
\begin{proof}
We show that for the family of instances (\ref{eq:c0}) - (\ref{eq:c4}),  $\frac{z^{\it{2}\textnormal{-}\mathcal{A}}}{z^{\mathcal{A}_2}}=\frac{M}{4}+1$  and thus $\it{2}\textnormal{-}\mathcal{A}$ can be an arbitrarily bad approximation of $\mathcal{A}_2$.

First, we show that $z^{\it{2}\textnormal{-}\mathcal{A}} = \frac{M}{4}+1$ by showing that the optimal point for the LP relaxation is also feasible for $\it{2}\textnormal{-}\mathcal{A}$. To conclude the proof, we show that $z^{\mathcal{A}_2}=z^{IP}$ by providing an upper bound on $z^{\mathcal{A}_2}$ coming from a particular selection of multipliers.

In the case of $\it{2}\textnormal{-}\mathcal{A}$, we verify that $(x,y_1,y_2,y_3)=\left( M+1,\frac{1}{2},\frac{1}{2},\frac{1}{2} \right)$ is in $P^I(e_{i_1},e_{i_2} )$, where $K:= \left\{e_{i_1},e_{i_2}\right\}$ corresponds to an arbitrary selection of two constraints, i.e. any $K \in {[4] \choose 2}$. 
Let $S_K'$ be those variables with negative coefficients that are present in the inequalities in $K$. If constraint (\ref{eq:c4}) is in $K$, let $l$ denote the smallest index in $S_K'$. Otherwise, let $l$ denote the index of the variable (out of $\left\{x_2,x_3,x_4\right\}$) that is present in both constraints $\left\{e_{i_1},e_{i_2}\right\}$ (note that there must always be one such index).
Then it can be verified that the points $(M+1,0,0,0)+e_{l}^\top$ and $(M+1,1,1,1)-e_l^\top$ are in $P^I(e_{i_1},e_{i_2} )$ and so is the midpoint $\left( M+1,\frac{1}{2},\frac{1}{2},\frac{1}{2} \right)$. The latter point has value: $M+1-\frac{M}{2}\frac{3}{2}=\frac{M}{4}+1$, thus, in terms of objective function value, $\it{2}\textnormal{-}\mathcal{A}$ does not provide any extra improvement when compared to the LP relaxation.

Now, we show that $z^{\mathcal{A}_2}=1$. Since $P^I\subset \mathcal{A}_2(P)$, we have that $z^{\mathcal{A}_2} \ge 1$, and by the definition of $\mathcal{A}_2$, we have that $z^{\mathcal{A}_2} \le z^{P^I(\lambda,\mu)} \quad\forall \lambda,\mu \in \R^4_+$. Consider $\bar{\lambda}=(1,1,0,0)$, $\bar{\mu}=(0,0,1,1)$ and $c^\top=\left(1,- \frac{M}{2},- \frac{M}{2}, - \frac{M}{2}\right)$, then the problem $\max\left\{c^\top x:\ x \in P^I(\bar{\lambda},\bar{\mu})\right\}$ corresponds to
\begin{align*}
\max &\quad x_1 - \frac{M}{2}x_2 - \frac{M}{2}x_3- \frac{M}{2}x_4 \\
\textup{s.t.} &\quad 2x_1 - 2M x_2- M x_3 - M x_4 \le 2 \\
&\quad 2x_1 \qquad\qquad- M x_3- M x_4 \le M+2 \\
&\quad x_1,x_2,x_3,x_4 \in \mathbb{Z}_+.
\end{align*}
Note that $x_3$ and $x_4$ have the same coefficients in the objective and in every constraint, thus by dropping $x_4$ and rearranging the constraints we obtain a problem with the same optimal objective function value
\begin{align*}
\max &\quad x_1 - \frac{M}{2}x_2 - \frac{M}{2}x_3 \\
\textup{s.t.} &\quad x_1\le 1 \qquad\qquad+M x_2+\frac{M}{2} x_3  \\
&\quad x_1\le 1+\frac{M}{2} \qquad\qquad\quad+\frac{M}{2} x_3  \\
&\quad x_1,x_2,x_3\in \mathbb{Z}_+.
\end{align*}
Now, a simple case analysis (below) for each value of $x_1$, shows that $z^{P^I(\bar{\lambda},\bar{\mu})}=1$. Let $z$ denote the best objective function value for each case:

\begin{itemize}
\item Case $\left(x_1=0\right)$: it is easy to see that $z\le 0$.
\item Case $\left(x_1=1\right)$: it is easy to see that $z\le 1$ (in fact, it is equal to 1 when $x_2=x_3=0$).
\item Case $\left(2\le x_1\le \frac{M}{2}+1\right)$: note that in this case the first constraint forces either $x_2$ or $x_3$ to be at least one. Thus $z\le \frac{M}{2}+1 -\frac{M}{2}=1$.
\item Case $\left(k\frac{M}{2}+2\le x_1\le (k+1)\frac{M}{2}+1,\mbox{ for $k\ge 1$ integer}\right)$: similar to the previous case. Now, since $x_1\ge k\frac{M}{2}+2$, the second constraint forces $x_3\ge k$, this together with the first constraint forces $x_2+x_3\ge k+1$. Then, $z\le (k+1)\frac{M}{2}+1 -\frac{M}{2}-k\frac{M}{2} =1$.
\end{itemize}\end{proof}

The proof of the second part of Theorem \ref{thm:bad2} is very similar to the proof of the second part of Theorem \ref{thm:bad} and therefore we do not present it here.  

\section*{Acknowledgements}
Santanu S. Dey would like to acknowledge the support of the NSF grant CMMI \#1562578. 
\bibliographystyle{plain}
\bibliography{stable}

\begin{thebibliography}{10}

\bibitem{AndersenLWW07}
Kent Andersen, Quentin Louveaux, Robert Weismantel, and Laurence~A. Wolsey.
\newblock Inequalities from two rows of a simplex tableau.
\newblock In Matteo Fischetti and David~P. Williamson, editors, {\em Integer
  Programming and Combinatorial Optimization, 12th International {IPCO}
  Conference, Ithaca, NY, USA, June 25-27, 2007, Proceedings}, volume 4513 of
  {\em Lecture Notes in Computer Science}, pages 1--15. Springer, 2007.

\bibitem{andersen2010zero}
Kent Andersen and Robert Weismantel.
\newblock Zero-coefficient cuts.
\newblock In {\em Integer Programming and Combinatorial Optimization}, pages
  57--70. Springer, 2010.

\bibitem{AverkovB15}
Gennadiy Averkov and Amitabh Basu.
\newblock Lifting properties of maximal lattice-free polyhedra.
\newblock {\em Math. Program.}, 154(1-2):81--111, 2015.

\bibitem{AverkovBP17}
Gennadiy Averkov, Amitabh Basu, and Joseph Paat.
\newblock Approximation of corner polyhedra with families of intersection cuts.
\newblock In Friedrich Eisenbrand and Jochen K{\"{o}}nemann, editors, {\em
  Integer Programming and Combinatorial Optimization - 19th International
  Conference, {IPCO} 2017, Waterloo, ON, Canada, June 26-28, 2017,
  Proceedings}, volume 10328 of {\em Lecture Notes in Computer Science}, pages
  51--62. Springer, 2017.

\bibitem{BasuBCM09}
Amitabh Basu, Pierre Bonami, G{\'{e}}rard Cornu{\'{e}}jols, and Fran{\c{c}}ois
  Margot.
\newblock On the relative strength of split, triangle and quadrilateral cuts.
\newblock In Claire Mathieu, editor, {\em Proceedings of the Twentieth Annual
  {ACM-SIAM} Symposium on Discrete Algorithms, {SODA} 2009, New York, NY, USA,
  January 4-6, 2009}, pages 1220--1229. {SIAM}, 2009.

\bibitem{BasuHK16}
Amitabh Basu, Robert Hildebrand, and Matthias K{\"{o}}ppe.
\newblock Light on the infinite group relaxation {I:} foundations and taxonomy.
\newblock {\em 4OR}, 14(1):1--40, 2016.

\bibitem{BasuHK16a}
Amitabh Basu, Robert Hildebrand, and Matthias K{\"{o}}ppe.
\newblock Light on the infinite group relaxation {II:} sufficient conditions
  for extremality, sequences, and algorithms.
\newblock {\em 4OR}, 14(2):107--131, 2016.

\bibitem{BasuHK17}
Amitabh Basu, Robert Hildebrand, and Matthias K{\"{o}}ppe.
\newblock Equivariant perturbation in gomory and johnson's infinite group
  problem - {III:} foundations for the k-dimensional case with applications to
  k=2.
\newblock {\em Math. Program.}, 163(1-2):301--358, 2017.

\bibitem{BasuHKM13}
Amitabh Basu, Robert Hildebrand, Matthias K{\"{o}}ppe, and Marco Molinaro.
\newblock A (k+1)-slope theorem for the k-dimensional infinite group
  relaxation.
\newblock {\em {SIAM} Journal on Optimization}, 23(2):1021--1040, 2013.

\bibitem{bodur2017lower}
Merve Bodur, Alberto Del~Pia, Santanu~S Dey, and Marco Molinaro.
\newblock Lower bounds on the lattice-free rank for packing and covering
  integer programs.
\newblock {\em arXiv preprint arXiv:1710.00031}, 2017.

\bibitem{bodur2016aggregation}
Merve Bodur, Alberto Del~Pia, Santanu~S. Dey, Marco Molinaro, and Sebastian
  Pokutta.
\newblock Aggregation-based cutting-planes for packing and covering integer
  programs.
\newblock {\em Mathematical Programming}, Sep 2017.

\bibitem{BorozanC09}
Valentin Borozan and G{\'{e}}rard Cornu{\'{e}}jols.
\newblock Minimal valid inequalities for integer constraints.
\newblock {\em Math. Oper. Res.}, 34(3):538--546, 2009.

\bibitem{CornuejolsM09}
G{\'{e}}rard Cornu{\'{e}}jols and Fran{\c{c}}ois Margot.
\newblock On the facets of mixed integer programs with two integer variables
  and two constraints.
\newblock {\em Math. Program.}, 120(2):429--456, 2009.

\bibitem{DeyR08}
Santanu~S. Dey and Jean{-}Philippe~P. Richard.
\newblock Facets of two-dimensional infinite group problems.
\newblock {\em Math. Oper. Res.}, 33(1):140--166, 2008.

\bibitem{DeyR10}
Santanu~S. Dey and Jean{-}Philippe~P. Richard.
\newblock Relations between facets of low- and high-dimensional group problems.
\newblock {\em Math. Program.}, 123(2):285--313, 2010.

\bibitem{DeyW10}
Santanu~S. Dey and Laurence~A. Wolsey.
\newblock Constrained infinite group relaxations of mips.
\newblock {\em {SIAM} Journal on Optimization}, 20(6):2890--2912, 2010.

\bibitem{fukasawa2011exact}
Ricardo Fukasawa and Marcos Goycoolea.
\newblock On the exact separation of mixed integer knapsack cuts.
\newblock {\em Mathematical programming}, 128(1-2):19--41, 2011.

\bibitem{LouveauxW08}
Quentin Louveaux and Robert Weismantel.
\newblock Polyhedral properties for the intersection of two knapsacks.
\newblock {\em Math. Program.}, 113(1):15--37, 2008.

\bibitem{marchand2001aggregation}
Hugues Marchand and Laurence~A Wolsey.
\newblock Aggregation and mixed integer rounding to solve mips.
\newblock {\em Operations research}, 49(3):363--371, 2001.

\bibitem{Meyer1974}
R.~R. Meyer.
\newblock On the existence of optimal solutions to integer and mixed-integer
  programming problems.
\newblock {\em Mathematical Programming}, 7(1):223--235, Dec 1974.

\bibitem{RichardD10}
Jean{-}Philippe~P. Richard and Santanu~S. Dey.
\newblock The group-theoretic approach in mixed integer programming.
\newblock In Michael J{\"{u}}nger, Thomas~M. Liebling, Denis Naddef, George~L.
  Nemhauser, William~R. Pulleyblank, Gerhard Reinelt, Giovanni Rinaldi, and
  Laurence~A. Wolsey, editors, {\em 50 Years of Integer Programming 1958-2008 -
  From the Early Years to the State-of-the-Art}, pages 727--801. Springer,
  2010.

\bibitem{weismantel19970}
Robert Weismantel.
\newblock On the 0/1 knapsack polytope.
\newblock {\em Mathematical Programming}, 77(3):49--68, 1997.

\bibitem{wolsey1975faces}
Laurence~A Wolsey.
\newblock Faces for a linear inequality in 0--1 variables.
\newblock {\em Mathematical Programming}, 8(1):165--178, 1975.

\bibitem{zemel1978lifting}
Eitan Zemel.
\newblock Lifting the facets of zero--one polytopes.
\newblock {\em Mathematical Programming}, 15(1):268--277, 1978.

\end{thebibliography}
\appendix
\section*{Proof of Propostion~\ref{prop:packcover}}
\begin{proof} 
The result for the case of packing polyhedron is shown in~\cite{bodur2016aggregation}. We prove the result for the case of covering polyhedron. Consider the LP relaxation:
\begin{eqnarray*}
\textup{max} & \sum_{j = 1}^n c_jx_j \\
\textup{s.t.} & \sum_{j = 1}^n a_{1j}x_j \geq b_1 \\
& \sum_{j = 1}^n a_{2j}x_j \geq b_2 \\
&x_j \geq 0 \ \forall j \in [n].
\end{eqnarray*}
In the LP optimal solution $x^*$, at most $2$ variables are non-zero. If only one variable is positive, then it is easy to verify the result by rounding up this solution to produce a IP feasible solution. Therefore, without loss of generality, let $x^*_1 > 0$ and $x^*_2> 0$. 
Let $y^*$ be an optimal dual solution. By strong duality and complimentary slackness, we have:
\begin{eqnarray}
b_1y^*_1 + b_2 y^*_2 = c_1 x^*_1 + c_2 x^*_2 \label{eq:cs0}\\
A_{11}y^*_1 + A_{21}y^*_2 = c_1 \label{eq:cs1}\\
A_{12}y^*_1 + A_{22}y^*_2 = c_2.\label{eq:cs2}
\end{eqnarray}
We consider six cases:
\begin{enumerate}
\item $x^*_1 \geq 1, x^*_2 \geq 1$: Clearly $(\lceil x^*_1 \rceil, \lceil x^*_2 \rceil, 0)$ is IP feasible. It is clear that $\frac{c_1\lceil x^*_1 \rceil + c_2\lceil x^*_2 \rceil }{c_1 x^*_1 + c_2 x^*_2} \leq 2$.
\item $0.4 \leq x^*_1 < 1, x^*_2 \geq 1$: Again, $(\lceil x^*_1 \rceil, \lceil x^*_2 \rceil, 0)$ is IP feasible.
\begin{eqnarray*}
\frac{c_1\lceil x^*_1 \rceil + c_2\lceil x^*_2 \rceil }{c_1 x^*_1 + c_2 x^*_2} \leq \textup{max}\left\{\frac{c_1}{c_1x^*_1}, \frac{c_2\lceil x^*_2 \rceil}{c_2 x^*_2} \right\} \leq \textup{max}\{2.5, 2\} = 2.5
\end{eqnarray*}
\item $x^*_1 \geq 1, 0.4 \leq x^*_2 < 1$: Same as above.
\item $x^*_1 < 0.4, x^*_2 \geq 1$: Since $A_{i1} \leq b$ for $i \in [2]$, we have that $A_{i2}x^*_2 \geq b_i ( 1 - x^*_1)$ for $i \in [2]$ and therefore $(0, \left\lceil \frac{x^*_2}{1 - x^*_1} \right\rceil, 0)$ is IP feasible. Now note that
\begin{eqnarray*}
\frac{c_2\left\lceil \frac{x^*_2}{1 - x^*_1} \right\rceil}{c_2 x^*_2}\leq\frac{\left\lceil (5/3)\cdot x^*_2\right\rceil}{x^*_2} < \left\{\begin{array}{ll} 2 & 1 \leq x^*_2 \leq 6/5 \\ 2.5 & 6/5 < x^*_2 \leq 9/5 \\ 5/3 + \frac{1}{x^*_2} \leq 20/9 &  x^*_2 \geq 9/5 \end{array}\right.
\end{eqnarray*}
\item $x^*_1 \geq 1, x^*_2 < 0.4$: Same as above.
\item $x^*_1 < 1, x^*_2 < 1$: In this case, 
\begin{eqnarray*}
\frac{c_1\lceil x^*_1 \rceil + c_2\lceil x^*_2 \rceil }{c_1 x^*_1 + c_2 x^*_2} = \frac{c_1 + c_2}{c_1 x^*_1 + c_2 x^*_2} & = & \frac{c_1}{c_1 x^*_1 + c_2 x^*_2} + \frac{c_2}{c_1 x^*_1 + c_2 x^*_2} \\
& \leq & \frac{A_{11}y^*_1 + A_{21}y^*_2}{b_1y^*_1 + b_2 y^*_2} + \frac{A_{12}y^*_1 + A_{22}y^*_2}{b_1y^*_1 + b_2 y^*_2}\quad (\textup{using } (\ref{eq:cs0}), (\ref{eq:cs1}), (\ref{eq:cs2})) \\
& \leq 2,
\end{eqnarray*}
where the last inequality follows from the fact that $A_{ij} \leq b_i$ for $i \in [2]$, $j \in [2]$.
\end{enumerate}
\end{proof}

\end{document}